\documentclass[12pt,english,reqno]{amsart}
\usepackage{amsmath, amsthm, amssymb}

\advance \topmargin by -\headheight
\advance \topmargin by -\headsep
     
\evensidemargin \oddsidemargin
\theoremstyle{definition}

\theoremstyle{remark}

\numberwithin{equation}{section}

\numberwithin{equation}{section} 
\numberwithin{figure}{section} 
\theoremstyle{plain}
\newtheorem{thm}{Theorem}[section]
  \theoremstyle{plain}
  \newtheorem{prop}[thm]{Proposition}
  \theoremstyle{plain}
  \newtheorem{lem}[thm]{Lemma}
  \theoremstyle{plain}
  \newtheorem{cor}[thm]{Corollary}

\numberwithin{equation}{section}
\numberwithin{thm}{section}

\usepackage{babel}

\begin{document}
\title[Waring's problem for cubes]{On Waring's problem: two cubes\\ and two minicubes}
\author[Siu-lun Alan Lee]{Siu-lun Alan Lee$^*$}
\address{School of Mathematics, University of Bristol, University Walk,
 Bristol BS8 1TW, United Kingdom}
\email{sl5701@bris.ac.uk}
\begin{abstract}
We establish that almost every positive integer $n$ is the sum of
four cubes, two of which are at most $n^{\theta}$, as long as $\theta\geq192/869$.
An asymptotic formula for the number of such representations is established
when $1/4<\theta<1/3$.
\end{abstract}
\thanks{The author is supported by the University of Bristol Overseas Centenary
Postgraduate Research Scholarship.}
\maketitle
\section{Introduction}
\par Davenport proved in \cite{davenport1939} that almost every
natural number can be expressed as a sum of four positive integral
cubes. It is now known that when $N$ is sufficiently large, the number
of positive integers at most $N$ that fail to be written in such
a way is slightly smaller than $N^{37/42}$. Since any integer congruent
to 4 (mod 9) is never a sum of three cubes, the number of summands
here cannot in general be reduced. A heuristic argument shows, however,
that one of the four cubes is almost redundant. This motivates the
work of Br\"{u}dern and Wooley (see \cite{BrudernWooley2009}) on the
representation of almost all positive integers as a sum of four cubes,
one of which is small (henceforth we call this a \textit{minicube}).
They have shown that such a minicube can be as small as $n^{5/36}$
without obstructing the existence of representations. This raises
the question as to whether we can restrict not only one, but two (or
even more) of the cubes in such representation to be minicubes, and
still get an almost all result. The purpose of this paper is to investigate
representations of natural numbers by sums of four cubes, two of which
are small.
\par When $n$ is a positive integer and $\theta>0$, write $r_{\theta}(n)$
for the number of integral solutions to the equation
\begin{equation}
n=x_{1}^{3}+x_{2}^{3}+y_{1}^{3}+y_{2}^{3}\label{eq:r(n)}\end{equation}
where $x_{1},x_{2},y_{1},y_{2}$ are natural numbers satisfying $y_{1},y_{2}\leq n^{\theta}$.
Plainly any one of these variables satisfying this equation must be
at most $n^{1/3}$, so a trivial upper bound for $\theta$ is $1/3$.
A formal application of the circle method suggests that
\[
r_{\theta}(n)\sim\frac{\Gamma(4/3)^{2}}{\Gamma(2/3)}\mathfrak{S}(n)n^{2\theta-1/3},\]
with $\mathfrak{S}(n)$ being the familiar singular series associated
with the representation of positive integers as sums of four cubes.
Recalling the estimate $\mathfrak{S}(n)\gg1$ (see Exercise 3 of section
4.6 of \cite{vaughan1997}), we therefore anticipate that $r_{\theta}(n)\geq1$
as long as $n$ is large enough and $\theta>1/6$. We establish this
for almost all $n$, in section \ref{sec:Existence}, for values of
$\theta$ rather smaller than $2/9$.
\begin{thm}
\label{thm:existence}Whenever $\theta\geq192/869$, we have that
$r_{\theta}(n)\geq1$ for almost all integers $n$.
\end{thm}
In some sense, the sum of two cubes and two minicubes at most $n^{\theta}$
employed in the representation \eqref{eq:r(n)} carry the same weight
as $2+6\theta$ cubes. Thus Theorem 1.1 asserts that almost every
natural number $n$ is the sum of at most $3.326$ cubes.
\par In section \ref{sec:An-asymptotic-formula}, we establish the asymptotic
formula for $r_{\theta}(n)$ in the following theorem.
\begin{thm}
\label{thm:asymptotic}Whenever $1/4<\theta<1/3$, the asymptotic
formula
\begin{equation}
r_{\theta}(n)=\frac{\Gamma(4/3)^{2}}{\Gamma(2/3)}\mathfrak{S}(n)n^{2\theta-1/3}+O(n^{2\theta-1/3}(\log n)^{-1})\end{equation}
holds for almost all positive integers $n$.
\end{thm}
\par This result can be compared with Br\"{u}dern and Wooley's result (see
Theorem 1.2 of \cite{BrudernWooley2009}) on representations as sums
of three cubes and a minicube. Our range of permissible values of
$\theta$ is identical to that obtained in the latter paper.
\par We establish Theorems \ref{thm:existence} and \ref{thm:asymptotic}
using the Hardy-Littlewood method. We begin in section \ref{sec:Existence}
by laying the foundations for the application of this method. This
leads to a lower bound for the contribution from the major arcs. Some
auxiliary mean value estimates vital to the proof of Theorem \ref{thm:existence}
are then introduced. Bessel's inequality is used to relate the exceptional
set to a minor arc estimate. Following three pruning processes, the
proof of Theorem \ref{thm:existence} is complete. The derivation
of the asymptotic formula in Theorem \ref{thm:asymptotic} is covered
in section \ref{sec:An-asymptotic-formula} and essentially follows
by conventional means. 
\par Throughout this paper, we use $ $$\varepsilon$ to denote an arbitrarily
small positive constant. The implicit constants in Vinogradov's well-known
notations $\ll$ and $\gg$ will depend at most on $\varepsilon$.
Whenever $\varepsilon$ appears in a statement, either implicitly
or explicitly, we assert that the statement is true for each $\varepsilon>0$.
Note that the 'value' of $\varepsilon$ will change from statement
to statement. The letter $\varpi$ always denotes a prime, and any
variable denoted by the letter $p$ (with or without subscripts) will
be a prime that is congruent to 2 (mod 3). As usual, write $e(z)=e^{2\pi iz}$.
\par The author would like to thank Trevor Wooley for his guidance and
comments during the course of this research.
\section{\label{sec:Existence}Existence of representations}
We begin our proof of Theorem \ref{thm:existence} by introducing
the basic ingredients for the application of the Hardy-Littlewood
method. 
\par Fix a large integer $N$. Let $\theta$ be a positive number with
$\theta\leq1/3$. Define
\begin{equation}
P=(N/4)^{1/3},\qquad R=P^{3\theta},\qquad Y=P^{11/79},\qquad L=(\log P)^{10}.\label{eq:parameters for existence}\end{equation}
We take $\eta$ to be a sufficiently small (but fixed) positive number,
and then define the set of smooth numbers
\begin{equation}
\mathcal{A}(R)=\{m\in[1,R]\cap\mathbb{Z}:\varpi\text{ prime and }\varpi|m\Rightarrow\varpi\leq R^{\eta}\}.\label{eq:smooth numbers}\end{equation}
Also, when $\alpha\in[0,1)$, define the
generating functions
\begin{equation}
f(\alpha)=\sum_{P<x\leq2P}e(\alpha x^{3})\qquad\text{and}\qquad h(\alpha)=\sum_{y\in\mathcal{A}(R)}e(\alpha y^{3}).\label{eq:f}\end{equation}
When $X$ and $Z$ are positive numbers, define
\[
\mathcal{A}^{*}(X,Z)=\{n\in\mathbb{Z}\cap[1,X]:\varpi|n\Rightarrow\varpi\leq Z^{\eta}\}.\]
Put $\mathcal{B}(X,Z)=\mathcal{A}^{*}(2X,Z)\backslash\mathcal{A}^{*}(X,Z)$.
Note that $\mathcal{A}(X)=\mathcal{A}^{*}(X,X)$. Fix $\tau>0$ with
the property that $\tau^{-1}>852+16\sqrt{2833}\approx1703.6$. Define
$J=\lfloor\frac{1}{2}\tau\log P\rfloor$, and when $\alpha\in\mathbb{R},$
write
\begin{equation}
K(\alpha)=\sum_{2^{-J}Y<p\leq Y}\sum_{w\in\mathcal{B}(P/p,2P/Y)}e(\alpha p^{3}w^{3}).\label{eq:K}\end{equation}
\par For all $\theta>0$ and integers $n$ with $N<n\leq2N$,
let $\rho_{\theta}(n)$ denote the number of integral solutions to
the equation
\begin{equation}
n=x^{3}+(pw)^{3}+y_{1}^{3}+y_{2}^{3},\label{eq:rho(n) for existence}\end{equation}
with
\[P<x\leq2P,\qquad2^{-J}Y<p\leq Y,\qquad w\in\mathcal{B}(P/p,2P/Y),\qquad y_{1},y_{2}\in\mathcal{A}(R).\]
It is apparent that $r_{\theta}(n)\geq\rho_{\theta}(n)$, and our
goal is to establish a lower bound for $\rho_{\theta}(n)$ that produces
the desired lower bound for $r_{\theta}(n)$. To this end, for any
measurable subset $\mathfrak{B}$ of $[0,1),$ define
\begin{equation}
\rho_{\theta}(n;\mathfrak{B})=\int_{\mathfrak{B}}f(\alpha)K(\alpha)h(\alpha)^{2}e(-n\alpha)\:\mathrm{d}\alpha.\label{eq:rho(n,B) for existence}\end{equation}
By orthogonality, we have $\rho_{\theta}(n)=\rho_{\theta}(n;[0,1))$
for all integers $n$ with $N<n\leq2N$. 
\par We analyse this integral using the Hardy-Littlewood method. When $a\in\mathbb{Z}$
and $q\in\mathbb{N}$ satisfy $0\leq a\leq q\leq L$ and $(a,q)=1$,
define
\begin{equation}
\mathfrak{P}(q,a)=\{\alpha\in[0,1):|\alpha-a/q|\leq LN^{-1}\}.\label{eq:P(q,a)}\end{equation}
In addition, for any positive number $X$, when $a\in\mathbb{Z}$ and $q\in\mathbb{N}$ satisfy $0\leq a\leq q\leq X$
and $(a,q)=1$, define
\begin{equation}
\mathfrak{M}(q,a;X)=\{\alpha\in[0,1):|q\alpha-a|\leq XP^{-3}\}.\label{eq:M(q,a;X)}\end{equation}
With this in mind, we define the \textit{major arcs} $\mathfrak{P}$
to be the union of the arcs $\mathfrak{P}(q,a)$ with $a\in\mathbb{Z}$,
$q\in\mathbb{N}$ satisfying $0\leq a\leq q\leq L$ and $(a,q)=1$.
Similarly, when $1\leq X\leq N^{1/2}$, define the major arcs $\mathfrak{M}(X)$
to be the union of the arcs $\mathfrak{M}(q,a;X)$ with $a\in\mathbb{Z}$
and $q\in\mathbb{N}$ satisfying $0\leq a\leq q\leq X$ and $(a,q)=1$.
Their respective complements in $[0,1)$ are the \textit{minor arcs}
$\mathfrak{p}$ and $\mathfrak{m}(X)$. The major arcs $\mathfrak{P}$
are of central interest in our argument, with $\mathfrak{M}(X)$ employed
as a tool for pruning the minor arc $\mathfrak{p}$ later.
\section{Major arc estimate}
The familiar approach to estimating the major arc contribution $\rho_{\theta}(n;\mathfrak{P})$,
which we largely follow, is to approximate the generating functions
in the integrand of \eqref{eq:rho(n,B) for existence} by some suitably
well-behaved functions. First, when $a\in\mathbb{Z}$ and $q\in\mathbb{N}$,
let
\begin{equation}
S(q,a)=\sum_{r=1}^{q}e(ar^3/q).\label{eq:S(q,a)}\end{equation}
Also, when $\beta$ is a real number and $Z$ is a positive number, write
\begin{equation}
v(\beta;Z)=\int_{Z}^{2Z}e(\beta\gamma^{3})\:\mathrm{d}\gamma.\label{eq:v(beta)}\end{equation}
In particular, write $v(\beta)$ for $v(\beta;P)$. Recall from Theorem 4.1 of \cite{vaughan1997} that when $\alpha\in\mathbb{R}$,
$a\in\mathbb{Z}$ and $q\in\mathbb{N}$, we have
\begin{equation}
f(\alpha)=q^{-1}S(q,a)v(\alpha-a/q)+O(q^{1/2+\varepsilon}(1+P^{3}|\alpha-a/q|)^{1/2}).\label{eq:approx. for f in general}\end{equation}
In particular, for all $\alpha\in\mathfrak{P}(q,a)\subseteq\mathfrak{P}$,
one obtains from \eqref{eq:P(q,a)} the relation
\begin{equation}
f(\alpha)=q^{-1}S(q,a)v(\alpha-a/q)+O(L^{1+\varepsilon}).\label{eq:approx for f on P(q,a)}\end{equation}
Similarly, it follows from Lemma 8.5 of \cite{wooley1991} that for
all $\alpha\in\mathfrak{P}(q,a)\subseteq\mathfrak{P}$, 
\begin{equation}
h(\alpha)=q^{-1}S(q,a)h(0)+O(RL^{-5}).\label{eq:approx for h on P(q,a)}\end{equation}
As in the argument on p.13 of \cite{BrudernWooley2009}, it is a consequence
of this lemma that there exists a positive constant $C$ with the
property that for all $\alpha\in\mathfrak{P}(q,a)\subseteq\mathfrak{P}$,
\begin{equation}
K(\alpha)=Cq^{-1}S(q,a)v(\alpha-a/q)+O(PL^{-5}).\label{eq:approx for K on P(q,a)}\end{equation}
\par When $\beta\in\mathbb{R}$, define
\begin{equation}
u(\beta)=Ch(0)^{2}v(\beta)^{2}.\label{eq:u(beta)}\end{equation}
Successive applications of \eqref{eq:approx for f on P(q,a)}, \eqref{eq:approx for h on P(q,a)}
and \eqref{eq:approx for K on P(q,a)} then yield the relation
\begin{equation}
f(\alpha)K(\alpha)h(\alpha)^{2}=\Big(q^{-1}S(q,a)\Big)^{4}u(\alpha-a/q)+O(P^{2}R^{2}L^{-5})\label{eq:fKh^2}\end{equation}
for all $\alpha\in\mathfrak{P}(q,a)$. For all positive integers $q$
and $n$, write
\begin{equation}
A(q,n)=\sum_{\substack{a=1\\
(a,q)=1}
}^{q}\Big(q^{-1}S(q,a)\Big)^{-4}e(-an/q).\label{eq:A(q,n)}\end{equation}
In addition, when $n$ is a natural number, define
\begin{equation}
\mathfrak{S}(n;L)=\sum_{1\leq q\leq L}A(q,n)\label{eq:truncated singular series}\end{equation}
and
\begin{equation}
J(n;L)=\int_{-L/N}^{L/N}u(\beta)e(-n\beta)\mathrm{\: d}\beta.\label{eq:truncated singular integral}\end{equation}
Note that the measure of $\mathfrak{P}$ is $O(L^{3}/N)$, so integrating
both sides of \eqref{eq:fKh^2} against $e(-n\alpha)$ over $\alpha\in\mathfrak{P}$
yields
\begin{equation}
\rho_{\theta}(n;\mathfrak{P})=\mathfrak{S}(n;L)J(n;L)+O(P^{-1}R^{2}L^{-2}).\label{eq:truncated asymptotic for existence}\end{equation}
\par Next recall the estimate
\begin{equation}
v(\beta)\ll P(1+P^{3}|\beta|)^{-1},\label{eq:estimate for v(beta)}\end{equation}
obtained via integration by parts. This ensures that the completed
singular integral
\begin{equation}
J(n)=\int_{-\infty}^{\infty}u(\beta)e(-n\beta)\:\mathrm{d}\beta\label{eq:J(n)}\end{equation}
converges absolutely and uniformly in $n$. Also, since $h(0)\ll R$,
we have
\begin{equation}
J(n)-J(n;L)\ll R^{2}P^{2}\int_{L/N}^{\infty}(1+P^{3}\beta)^{-2}\:\mathrm{d}\beta\ll R^{2}P^{-1}L^{-1}.\label{eq:tail of singular integral}\end{equation}
Finally, the value of the singular integral can be computed to be
\begin{equation}
J(n)=Ch(0)^{2}\frac{\Gamma(4/3)^{2}}{\Gamma(2/3)}n^{-1/3},\label{eq:value of J(n)}\end{equation}
in accordance with the methods outlined on pages 21 and 22 of \cite{davenport2005}.
\par Meanwhile, Theorem 4.3 of \cite{vaughan1997} ensures that the singular
series
\begin{equation}
\mathfrak{S}(n)=\sum_{q=1}^{\infty}A(q,n)\label{eq:singular series}\end{equation}
converges absolutely and uniformly in $n$. Also, equation (1.3) of
\cite{kawada1996} shows that $1\ll\mathfrak{S}(n)\ll(\log\log n)^{4}$.
In addition, the argument on p.14 of \cite{BrudernWooley2009} demonstrates
that
\begin{equation}
\mathfrak{S}(n)-\mathfrak{S}(n;L)\ll L^{-1/16}\label{eq:tail of singular series}\end{equation}
for all but $O(NL^{-1/16})$ integers $n$ with $N<n\leq2N$. Finally,
$|\mathfrak{S}(n;L)|\gg1$ for all but $O(NL^{-1/16})$ integers $n$
with $N<n\leq2N$. 
\par Equations \eqref{eq:truncated asymptotic for existence}, \eqref{eq:tail of singular integral},
\eqref{eq:tail of singular series} and \eqref{eq:value of J(n)}
together thus lead to the asymptotic lower bound
\begin{equation}
\rho_{\theta}(n;\mathfrak{P})\gg\mathfrak{S}(n)n^{2\theta-1/3}+O(n^{2\theta-1/3}(\log n)^{-1/16}),\label{eq:asymptotic for existence}\end{equation}
valid for all integers $n$ with $N<n\leq2N$, with at most $O(N(\log N)^{-1/16})$
exceptions. We summarise this conclusion in the following proposition.
\begin{prop}
\label{pro:major arc estimate for existence}For all but $O(N(\log N)^{-1/16})$
integers $n$ with $N<n\leq2N$, we have $\rho_{\theta}(n;\mathfrak{P})\gg n^{2\theta-1/3}$.
\end{prop}
\section{Auxiliary estimates}
We now establish several mean value estimates of generating functions
that are required in the evaluation of the minor arc contribution
$\rho_{\theta}(n;\mathfrak{p})$ in the following section. 
\par When $\alpha\in\mathbb{R}$ and $Q\geq1$, write
\begin{equation}
g(\alpha)=\sum_{Q<y\leq2Q}e(\alpha y^{3}).\label{eq:definition of g}\end{equation}
We record for future reference the following lemma.
\begin{lem}
\label{lem:mean value of g^2h^6}Whenever $R\leq Q^{2/3}$, we have
\[
\int_{0}^{1}|g(\alpha)^{2}h(\alpha)^{6}|\:\mathrm{d}\alpha\ll QR^{13/4-\tau}.\]
\end{lem}
\begin{proof}
This is Lemma 2.1 of \cite{BrudernWooley2009}.
\end{proof}
This gives rise to the following corollaries.
\begin{cor}
\label{lem:mean value of f^2h^6}Let
\[
T_{1}=\int_{0}^{1}|f(\alpha)^{2}h(\alpha)^{6}|\;\mathrm{d}\alpha\qquad\text{and}\qquad T_{2}=\int_{0}^{1}|K(\alpha)^{2}h(\alpha)^{6}|\;\mathrm{d}\alpha\]
Then whenever $R\leq(P/2)^{2/3}$, we have
\[
T_{1}\ll PR^{13/4-\tau}\qquad\text{and}\qquad T_{2}\ll PR^{13/4-\tau}.\]
\end{cor}
\begin{proof}
The first inequality is immediate from Lemma \ref{lem:mean value of g^2h^6}
on taking $Q=P/2$. To estimate the latter mean value, observe that
if $P/2<m\leq P$, and $m=pw$ for some prime $p$ and integer $w$
occuring in the summations of \eqref{eq:K}, then since $2^{-J}Y>R^{\eta},$
the pair $(p,w)$ is uniquely defined. Hence, by orthogonality, it
follows on considering the underlying diophantine equations that $T_{2}\leq T_{1}$.
The required result thus again follows from Lemma \eqref{lem:mean value of g^2h^6}.
\end{proof}
We also quote the following useful lemma.
\begin{lem}
\label{lem:mean value of f^2g^4}Whenever $1\leq Q\leq P$, we have
\[
\int_{0}^{1}|f(\alpha)^{2}g(\alpha)^{4}|\:\mathrm{d}\alpha\ll P^{\varepsilon}(PQ^{2}+P^{-1}Q^{9/2}).\]
\end{lem}
\begin{proof}
This is the first estimate of Lemma 2.3 of \cite{BrudernWooley2009}.
\end{proof}
The following result is a direct consequence of this lemma.
\begin{cor}
\label{lem:estimate for K^2h^4}Whenever $R\leq(P/2)^{2/3}$, we have
the estimate
\[
\int_{0}^{1}|K(\alpha)^{2}h(\alpha)^{4}|\:\mathrm{d}\alpha\ll P^{1+\varepsilon}R^{2}.\]
\end{cor}
\begin{proof}
As in the argument of the proof of Corollary \ref{lem:mean value of f^2h^6},
we have
\[
\int_{0}^{1}|K(\alpha)^{2}h(\alpha)^{4}|\:\mathrm{d}\alpha\leq\int_{0}^{1}|f(\alpha)^{2}g(\alpha)^{4}|\:\mathrm{d}\alpha.\]
The desired conclusion is thus immediate from Lemma \ref{lem:mean value of f^2g^4}.
\end{proof}
\par Next we define the mean value
\[ U_1=\int_{0}^{1}|K(\alpha)|^8\;\mathrm{d}\alpha.\]
By considering the underlying diophantine equation, it follows from Theorem 2 of \cite{vaughan1986} that
\begin{equation}
U_1\ll P^{5}.\label{eq:|K|^8}\end{equation}
\par Introduce the function $f^{*}:[0,1)\rightarrow\mathbb{C}$ given by
\begin{equation}
f^{*}(\alpha)=\begin{cases}
q^{-1}S(q,a)v(\alpha-a/q),\qquad & \text{ when }\alpha\in\mathfrak{M}(q,a;P^{6/5})\subseteq\mathfrak{M}(P^{6/5}),\\
0, & \text{ otherwise}.\end{cases}\label{eq:f*}\end{equation}
Also, when $R\leq X\leq P^{6/5}$, let
\[U(X)=\int_{\mathfrak{M}(2X)\backslash\mathfrak{M}(X)}|f^*(\alpha)|^8\;\mathrm{d}\alpha. \] 
Finally, define
\[U_2=\int_{\mathfrak{M}(R)\backslash\mathfrak{P}}|f^*(\alpha)|^{16/3}\;\mathrm{d}\alpha.\]
Then we have the following estimates.
\begin{lem}
\label{lem:mean of f*^8 and f*^16/3}Whenever $R\leq X\leq P^{6/5},$
we have
\[
U(X)\ll P^{5}X^{\varepsilon-4/3}\qquad\text{and}\qquad U_2\ll P^{7/3}L^{-4/9}.\]
\end{lem}
\begin{proof}
These two inequalities are established by the argument of Lemma 5.1 of \cite{vaughan1989}.\end{proof}
When $\beta\in\mathbb{R}$ and $Z$ is a positive number, write
\begin{equation}
w(\beta;Z)=\int_{0}^{Z}e(\beta\gamma^3)\;\mathrm{d}\gamma.\label{eq:w}\end{equation}
Also, when $\alpha\in\mathbb{R}$ and $1<Z\leq R$, let
\begin{equation}
F^*(\alpha,Z)=\begin{cases}
q^{-1}S(q,a)w(\alpha-a/q;2P),\qquad & \text{when }\alpha\in\mathfrak{M}(q,a;Z),\\
0,\qquad & \text{otherwise.}\end{cases}\label{eq:F*}\end{equation}
When $\alpha\in\mathbb{R}$ and $\mathcal{B}\subseteq[1,R]$, let
\begin{equation}
j(\alpha;\mathcal{B})=\sum_{x\in\mathcal{B}}e(\alpha x^3).\label{eq:j(alpha)}\end{equation}
Finally, when $t>0$, $1<Z\leq R$ and $\mathcal{B}\subseteq[1,R]$, put
\begin{equation} U_3(t;Z;\mathcal{B})=\int_{\mathfrak{M}(Z)}|f^*(\alpha)^tj(\alpha;\mathcal{B})^6|\;\mathrm{d}\alpha\label{eq:U_3}\end{equation}
and
\begin{equation} U_4(t;Z;\mathcal{B})=\int_{\mathfrak{M}(Z)}|F^*(\alpha;Z)^{t}j(\alpha;\mathcal{B})^6|\;\mathrm{d}\alpha.\label{eq:U_4}\end{equation}
We give upper bounds for these two integrals in the following lemma.
\begin{lem}
\label{lem:mean of f*^vh^6}Let $1<Z\leq R$ and $\mathcal{B}\subseteq[1,R]$. Then when $t>2$, we have
\[ U_3(t;Z;\mathcal{B})\ll P^{t-3}R^6.\]
On the other hand, when $t>3$, we have the same upper bound for $U_4(t;Z;\mathcal{B})$.
\end{lem}
\begin{proof}
This argument is largely akin to that given in the proof of Lemma
5.4 of \cite{VaughanWooley2000}. Define the arithmetic function $w$
multiplicatively by
\begin{equation}
w(\varpi^{3u+v})=\begin{cases}
3\varpi^{-u-1/2},\qquad & \text{if }u\geq0\text{ and }v=1,\\
\varpi^{-u-1}, & \text{if }u\geq0\text{ and }v\in\{2,3\}.\end{cases}\label{eq:arithmetic function w}\end{equation}
Then from Lemmata 4.3, 4.4 and Theorem 4.2 of \cite{vaughan1997},
we deduce that when $(a,q)=1$, we have
\begin{equation}
q^{-1}S(q,a)\ll w(q).\label{eq:bounding q^-1S(q,a) by w(q)}\end{equation}
\par When $\alpha\in\mathbb{R}$, $1<Z\leq R$, $\theta_1>2$ and $\theta_2>1$, let
\begin{equation}
\Upsilon(\alpha;\theta_1,\theta_2;Z)=\begin{cases}
w(q)^{\theta_1}(1+P^{3}|\alpha-a/q|)^{-\theta_2},\qquad & \text{when }\alpha\in\mathfrak{M}(q,a;Z),\\
0, & \text{otherwise.}\end{cases}\label{eq:upsilon}\end{equation}
With $\theta_1,\theta_2$ and $Z$ as above and $\mathcal{B}\subseteq[1,R]$, write
\begin{equation}
T(\theta_1,\theta_2;Z,\mathcal{B})=\int_{\mathfrak{M}(Z)}|\Upsilon(\alpha;\theta_1,\theta_2;Z)j(\alpha;\mathcal{B})^{6}|\:\mathrm{d}\alpha.\label{eq:T in terms of upsilon}\end{equation}
Recall the definition \eqref{eq:f*} of $f^{*}$. The estimates \eqref{eq:bounding q^-1S(q,a) by w(q)}
and \eqref{eq:estimate for v(beta)} imply that when $t>0$, we have
\begin{equation}
\int_{\mathfrak{M}(Z)}|f^{*}(\alpha)^tj(\alpha;\mathcal{B})^{6}|\:\mathrm{d}\alpha\ll P^tT(t,t;Z;\mathcal{B}).\label{eq:integral of f*h in terms of upsilon}\end{equation}
On the other hand, Theorem 7.3 of \cite{vaughan1997} yields
\begin{equation}
w(\beta;U)\ll U(1+U^{3}|\beta|)^{-1/3}\label{eq:estimate for w_0}\end{equation}
for all positive numbers $U$. This, together with \eqref{eq:F*}, \eqref{eq:bounding q^-1S(q,a) by w(q)}, \eqref{eq:upsilon} and \eqref{eq:T in terms of upsilon}, give rise to
\begin{equation}
\int_{\mathfrak{M}(Z)}|F^*(\alpha;Z)^tj(\alpha;\mathcal{B})^6|\;\mathrm{d}\alpha\ll P^tT(t,t/3;Z;\mathcal{B}).\label{eq:integral of F*t in terms of upsilon}\end{equation}
We therefore require an upper bound for $T$. Here we prove that when $\theta_1>2$, $\theta_2>1$, $1<Z\leq R$ and $\mathcal{B}\subseteq[1,R]$, we have
\begin{equation}
T(\theta_1,\theta_2;Z;\mathcal{B})\ll P^{-3}R^6\label{eq:ultimate bound for T}\end{equation}
\par Substituting \eqref{eq:upsilon} into \eqref{eq:T in terms of upsilon}
yields
\[
T(\theta_1,\theta_2;Z;\mathcal{B})\leq\sum_{1\leq q\leq Z}w(q)^{\theta_1}\sum_{a=1}^{q}\int_{\mathfrak{M}(q,a;Z)}(1+P^{3}|\alpha-a/q|)^{-\theta_2}\Big|\sum_{y\in\mathcal{B}}e(y^{3}\alpha)\Big|^{6}\:\mathrm{d}\alpha,\]
where we have removed the coprimality condition $(a,q)=1$ in the
$a$-summation. Recalling the definition \eqref{eq:M(q,a;X)} of the
major arcs $\mathfrak{M}(q,a;Z)$ and making the change of variables
$\alpha=a/q+\beta$ in the integral, we obtain
\begin{equation}
T(\theta_1,\theta_2;Z;\mathcal{B})\leq\sum_{1\leq q\leq Z}w(q)^{\theta_1}\int_{-\infty}^{\infty}(1+P^3|\beta|)^{-\theta_2}\sum_{a=1}^{q}\Big|\sum_{y\in\mathcal{B}}e\Big(y^{3}(\beta+a/q)\Big)\Big|^{6}\;\mathrm{d}\beta.\label{eq:first upper bound for T}\end{equation}
Here we have extended the range of integration from $|\beta|\leq q^{-1}ZP^{-3}$
to the whole real line. This is valid as the completed integral evidently converges absolutely.
\par When $\boldsymbol{y}=(y_{1},...,y_{6})$
with integral coordinates $y_{1},...,y_{6}$, write
\begin{equation}
\Psi(\boldsymbol{y})=y_{1}^{3}-y_{2}^{3}+y_{3}^{3}-y_{4}^{3}+y_{5}^{3}-y_{6}^{3}.\label{eq:phi(y)}\end{equation}
Expanding the innermost sum in \eqref{eq:first upper bound for T}
and then swapping the $a$- and $\boldsymbol{y}$-sums, we see that 
\[
\sum_{a=1}^{q}\Big|\sum_{y\in\mathcal{B}}e\Big(y^{3}(\beta+a/q)\Big)\Big|^{6}=\sum_{y_{1},...,y_{6}\in\mathcal{B}}e(\beta\Psi(\boldsymbol{y}))\sum_{a=1}^{q}e(a\Psi(\boldsymbol{y})/q).\]
The $a$-sum here is zero unless $q|\Psi(\boldsymbol{y})$, in which
case it equals $q$. Thus
\begin{equation}
\sum_{a=1}^{q}\Big|\sum_{y\in\mathcal{B}}e\Big(y^{3}(\beta+a/q)\Big)\Big|^{6}=q\sum_{\substack{y_{1},...,y_{6}\in\mathcal{B}\\
q|\Psi(\boldsymbol{y})}
}e(\beta\Psi(\boldsymbol{y})).\label{eq:sixth moment sum in terms of psi}\end{equation}
Let $\rho(q)$ be the number of solutions to the congruence $\Psi(\boldsymbol{y})\equiv0\;(\text{mod }q)$,
under the constraint that each coordinate of $\boldsymbol{y}$ is
a positive integer not exceeding $q$. A trivial estimate then gives
\begin{equation}
\sum_{\stackrel{y_{1},...,y_{6}\in\mathcal{B}}{q|\Psi(\boldsymbol{y})}}e(\beta\Psi(\boldsymbol{y}))\leq\sum_{\stackrel{1\leq y_{1},...,y_{6}\leq R}{q|\Psi(\boldsymbol{y})}}1\ll(R/q+1)^{6}\rho(q).\label{eq:bounding y-sum in terms of rho(q)}\end{equation}
By orthogonality, it follows from the definitions \eqref{eq:S(q,a)}
and \eqref{eq:phi(y)}, of $S(q,a)$ and $\Psi(\boldsymbol{y})$ respectively,
that
\[
q\rho(q)=\sum_{a=1}^{q}|S(q,a)|^{6}=\sum_{a=1}^{q}(q,a)^{6}\Big|S\Big(\frac{q}{(q,a)},\frac{a}{(q,a)}\Big)\Big|^{6}.\]
Hence Theorem 4.2 of \cite{vaughan1997} yields
\begin{equation}
q\rho(q)\ll\sum_{a=1}^{q}(q,a)^{6}((q/(q,a))^{2/3})^{6}=q^{4}\sum_{a=1}^{q}(q,a)^{2}\ll q^{6}.\label{eq:bounding rho(q)}\end{equation}
\par Putting \eqref{eq:sixth moment sum in terms of psi},
\eqref{eq:bounding y-sum in terms of rho(q)} and \eqref{eq:bounding rho(q)}
together, we see that the double sum within the integral in \eqref{eq:first upper bound for T}
has the asymptotic upper bound
\[
q(R/q+1)^{6}\rho(q)\ll(R/q+1)^{6}q^{6}=(R+q)^{6}\ll R^{6}.\]
Inserting this estimate into \eqref{eq:first upper bound for T},
we obtain
\begin{equation}
T(\theta_1,\theta_2;Z,\mathcal{B})\ll P^{-3}R^{6}\sum_{1\leq q\leq Z}w(q)^{\theta_1}.\label{eq:second upper bound for T}\end{equation}
It thus remains to show that the sum here is uniformly bounded over $Z$.
\par Now that $w$ is a multiplicative function, it suffices to evaluate the values
of $w(\varpi^{3u+v})^{\theta_1}$ for all primes $\varpi$ in the analysis
of the $q$-sum in \eqref{eq:second upper bound for T}. From the
definition \eqref{eq:arithmetic function w} of $w$, we readily confirm
that for any prime $\varpi$, we have the bounds
\[
w(\varpi^{3u+v})^{\theta_1}\ll_{\theta_1}\begin{cases}
\varpi^{-\theta_1u-\theta_1/2},\qquad & \text{if }u\geq0\mbox{\text{ and }}v=1,\\
\varpi^{-\theta_1u-\theta_1}, & \text{if }u\geq0\text{ and }v\in\{2,3\}.\end{cases}\]
This reveals that for each $\varpi$, we have
\[
\sum_{h=1}^{\infty}w(\varpi^{h})^{\theta_1}\ll_{\theta_1}\varpi^{-\theta_1/2}.\]
This ensures the existence of a positive number $A=A(\theta_1)$ for which 
the $q$-sum in \eqref{eq:second upper bound for T} is bounded above by the
Euler product
\[
\prod_{\varpi\leq Z}\Big(1+\sum_{h=1}^{\infty}w(\varpi^{h})^{\theta_1}\Big)\ll\prod_{\varpi\leq Z}(1+A\varpi^{-\theta_1/2})\ll\prod_{\varpi}(1+\varpi^{-\theta_1/2})^A\ll_{\theta_1}1.\]
The given condition $\theta_1>2$ validates the last inequality above. The desired inequality \eqref{eq:ultimate bound for T} follows immediately.
\par The lemma thus follows by putting together the inequalities \eqref{eq:integral of f*h in terms of upsilon}, \eqref{eq:integral of F*t in terms of upsilon} and \eqref{eq:ultimate bound for T}, and recalling the definitions \eqref{eq:U_3} and \eqref{eq:U_4} of $U_3$ and $U_4$ respectively.
\end{proof}
\section{Minor arc estimate}
On recalling \eqref{eq:parameters for existence}, it follows from
Proposition \ref{pro:major arc estimate for existence} that for almost
all integers $n$ with $N<n\leq2N$, we have
\[
\rho_{\theta}(n;\mathfrak{P})\gg P^{-1}R^{2}.\]
We therefore seek to show that the minor arc contribution $\rho_{\theta}(n;\mathfrak{p})$
is $o(P^{-1}R^{2})$ for almost all such $n$. For any measurable
subset $\mathfrak{B}$ of $[0,1)$, write
\begin{equation}
S(\mathfrak{B})=\sum_{N<n\leq2N}|\rho_{\theta}(n;\mathfrak{B})|^{2}.\label{eq:S(n,B)}\end{equation}
The desired bound for $\rho_{\theta}(n;\mathfrak{p})$ follows if
we can establish the relation
\begin{equation}
S(\mathfrak{p})=o(PR^{4}).\label{eq:aim of minor arc estimate for existence}\end{equation}
First introduce the arcs
\begin{align}
\mathfrak{\mathfrak{m}=\mathfrak{m}}(PY^{3})&,\qquad\mathfrak{D}=\mathfrak{M}(PY^{3})\backslash\mathfrak{M}(P^{6/5}),\nonumber\\ \mathfrak{U}=\mathfrak{M}(P^{6/5})\backslash\mathfrak{M}(R)&,\qquad\mathfrak{A=}\mathfrak{M}(R)\backslash\mathfrak{P}.\label{eq:prunings for existence}\end{align}
Then evidently $\mathfrak{p}=\mathfrak{m}\cup\mathfrak{D}\cup\mathfrak{U}\cup\mathfrak{A}$,
so 
\begin{equation}
S(\mathfrak{p})=S(\mathfrak{m})+S(\mathfrak{D})+S(\mathfrak{U})+S(\mathfrak{A}).\label{eq:splitting for existence}\end{equation}
For any positive number $Y$, define
\[
I(Y)=\int_{\mathfrak{m}}|f(\alpha)^{2}K(\alpha)^{6}|\:\mathrm{d}\alpha.\]
It is then a consequence of Corollary 3.2 of \cite{BrudernWooley2009}
that when $Y$ is chosen as in \eqref{eq:parameters for existence},
we have the bound
\begin{equation}
I(Y)\ll P^{19/4-\tau/2}Y^{-3/4}.\label{eq:I(Y)}\end{equation}
\begin{prop}
\label{pro:S(m)}As long as $\theta\leq2/9$, we have
\[
S(\mathfrak{m})\ll P^{175/79-\tau/6}R^{13/6-2\tau/3}.\]
\end{prop}
\begin{proof}
Applications of Bessel's inequality followed by H\"{o}lder's inequality
reveal that
\begin{equation}
S(\mathfrak{m})\leq\int_{\mathfrak{m}}|f(\alpha)^{2}K(\alpha)^{2}h(\alpha)^{4}|\:\mathrm{d}\alpha\leq I(Y)^{1/3}\Big(\int_{0}^{1}|f(\alpha)^{2}h(\alpha)^{6}|\:\mathrm{d}\alpha\Big)^{2/3}.\end{equation}
The restriction $\theta\leq2/9$ enables the application of Corollary
\ref{lem:mean value of f^2h^6} here. Together with \eqref{eq:I(Y)} and \eqref{eq:parameters for existence}, this gives
\[
S(\mathfrak{m})\ll(P^{19/4-\tau/2}Y^{-3/4})^{1/3}(PR^{13/4-\tau})^{2/3}
,\]
 and the desired conclusion follows from a modest calculation.\end{proof}
Note that when $\theta\geq192/869$, the bound provided by this proposition is indeed $o(PR^{4})$.
\par Next we evaluate $S(\mathfrak{D})$. For any non-negative integer
$l$, if the dyadic interval $(2^{l}P^{6/5},2^{l+1}P^{6/5}]$ lies
within the interval $(P^{6/5},PY^{3}]$, then $l$ satisfies the inequality
$0\leq l\leq c\log P$, where $c=86/(395\log2)$. By introducing another
dissection in the shape
\begin{equation}
\mathfrak{k}(X)=\mathfrak{M}(2X)\backslash\mathfrak{M}(X),\label{eq:k(X)}\end{equation}
we can thus split $\mathfrak{D}$ into the disjoint union
\begin{equation}
\mathfrak{D}=\bigcup_{0\leq l\leq c\log P}\mathfrak{k}(2^{l}P^{6/5}).\label{eq:splitting D}\end{equation}
Whence it suffices to consider $S(\mathfrak{k}(X))$ when $X\in(P^{6/5},PY^{3}]$.
With $\lambda=3/34-\tau/4$ and $X$ as such, we record for future
reference the bound
\begin{equation}
\int_{\mathfrak{k}(X)}|f(\alpha)^{2}K(\alpha)^{5}|\:\mathrm{d}\alpha\ll P^{4+\lambda+\varepsilon}Y^{-1-\lambda}(PY^{3}X^{-1})^{1/2},\label{eq:estimate for f^2K^5}\end{equation}
which is provided by equation (5.6) of \cite{BrudernWooley2009}.
This inequality yields the following bound for $S(\mathfrak{k}(X))$.
\begin{lem}
\label{lem:estimate for S(k(X))}Whenever $P^{6/5}<X\leq PY^{3}$ and $\theta\leq2/9$,
we have
\[
S(\mathfrak{k}(X))\ll P^{13/6+\lambda/3+\varepsilon}Y^{1/6-\lambda/3}R^{13/6-2\tau/3}.\]
\end{lem}
\begin{proof}
Applications of Bessel's inequality and H\"{o}lder's inequality reveal
that
\begin{align}
S(\mathfrak{k}(X))&\leq\int_{\mathfrak{k}(X)}|f(\alpha)^{2}K(\alpha)^{2}h(\alpha)^{4}|\:\mathrm{d}\alpha\label{eq:Bessel on S(k(X))}\\
&\leq\Big(\sup_{\alpha\in\mathfrak{k}(X)}|f(\alpha)|\Big)^{1/3}\int_{0}^{1}|f(\alpha)^{5/3}K(\alpha)^{2}h(\alpha)^{4}|\:\mathrm{d}\alpha.\label{eq:Holder on S(k(X))}\end{align}
On recalling that $X\geq P^{6/5}$, successive applications of \eqref{eq:approx. for f in general},
\eqref{eq:estimate for v(beta)} and Theorem 4.2 of \cite{vaughan1997}
give rise to the bound
\begin{equation}
\sup_{\alpha\in\mathfrak{k}(X)}|f(\alpha)|\ll PX^{-1/3}+X^{1/2+\varepsilon}\ll X^{1/2+\varepsilon}.\label{eq:estimate for f on k(X)}\end{equation}
This together with another use of H\"{o}lder's inequality on \eqref{eq:Holder on S(k(X))}
lead to
\begin{equation}
S(\mathfrak{k}(X))\ll\Big(X^{1/2+\varepsilon}\int_{\mathfrak{k}(X)}|f(\alpha)^{2}K(\alpha)^{5}|\:\mathrm{d}\alpha\Big)^{1/3}T_{2}^{1/6}T_{1}^{1/2}.\label{eq:Holder on S(k(X))}\end{equation}
Applications of \eqref{eq:estimate for f^2K^5} as well as Corollary
\ref{lem:mean value of f^2h^6} thus yield 
\[
S(\mathfrak{k}(X))\ll P^{\varepsilon}(X^{1/2}P^{4+\lambda}Y^{-1-\lambda}(PY^{3}X^{-1})^{1/2})^{1/3}(PR^{13/4-\tau})^{2/3}.\]
A modicum of computation confirms that this is indeed the bound in the statement of the lemma.
\end{proof}
The splitting in \eqref{eq:splitting D} reveals that
\[
S(\mathfrak{D})=\sum_{0\leq l\leq c\log P}S(\mathfrak{k}(2^{l}P^{6/5}))\ll P^{13/6+\lambda/3}Y^{1/6-\lambda/3}R^{13/6}.\]
With reference to \eqref{eq:parameters for existence} and the value
of $\lambda$, we have the following result.
\begin{prop}
\label{pro:S(D)}Provided that $\theta\leq2/9$, we have
\[
S(\mathfrak{D})\ll P^{175/79-17\tau/237}R^{13/6}.\]
\end{prop}
The reader can check that this bound is $o(PR^4)$ when $\theta\geq192/869$.
\par The treatment of $S(\mathfrak{U})$ is similar. For any non-negative
integer $l$, if 
\[(2^{-l}P^{6/5},2^{-l+1}P^{6/5}]\subseteq(R,P^{6/5}],\]
then $l$ satisfies the constraint $R\leq2^{-l}P^{6/5}\leq P^{6/5}$.
This implies that $0\leq l\leq c'\log P$, where $c'=6/(5\log2)$.
It follows that
\begin{equation}
\mathfrak{U}\subseteq\bigcup_{0\leq l\leq c^{'}\log P}\mathfrak{k}(2^{-l}P^{6/5}),\label{eq:splitting for U}\end{equation}
We therefore need a bound for $S(\mathfrak{k}(X))$ in the case where $R\leq X\leq P^{6/5}$. This is provided by the following lemma.
\begin{lem}
\label{lem:estimate for S(k(X)) (extended)}Whenever $R\leq X\leq P^{6/5}$
and $\theta\leq2/9$, we have
\[
S(\mathfrak{k}(X))\ll P^{7/3}R^{13/6-2\tau/3}X^{\varepsilon-1/3}+XP^{1+\varepsilon}R^2.\]\end{lem}
\begin{proof}
By \eqref{eq:approx. for f in general}, when $\alpha\in\mathfrak{k}(X)$,
we have
\begin{equation}
f(\alpha)=f^{*}(\alpha)+O(X^{1/2+\varepsilon}).\label{eq:approx for f on k(X)}\end{equation}
Inequality \eqref{eq:Bessel on S(k(X))} followed by this estimate implies that
\begin{equation}
S(\mathfrak{k}(X))\ll I_{1}(X)+X^{1+\varepsilon}\int_{0}^{1}|K(\alpha)^{2}h(\alpha)^{4}|\:\mathrm{d}\alpha,\label{eq:splitting up S(k(X))}\end{equation}
where
\begin{equation}
I_{1}(X)=\int_{\mathfrak{k}(X)}|f^{*}(\alpha)^{2}K(\alpha)^{2}h(\alpha)^{4}|\:\mathrm{d}\alpha.\label{eq:I_1}\end{equation}
Corollary \ref{lem:estimate for K^2h^4} implies that when $\theta\leq2/9$
and $X\leq P^{6/5}$, the second term in \eqref{eq:splitting up S(k(X))}
is $O(XP^{1+\varepsilon}R^{2})$. Meanwhile, an application of H\"{o}lder's inequality yields
\[
I_{1}(X)\leq T_{2}^{2/3}U_1^{1/12}U(X)^{1/4}.\]
Successive applications of Lemmata \ref{lem:mean of f*^8 and f*^16/3},
\ref{lem:mean value of f^2h^6} and \eqref{eq:|K|^8} give
rise to
\begin{align*}
I_{1}(X) & \ll(PR^{13/4-\tau})^{2/3}(P^{5})^{1/12}(P^{5}X^{\varepsilon-4/3})^{1/4}\\
 & =P^{7/3}R^{13/6-2\tau/3}X^{\varepsilon-1/3}.\end{align*}
The lemma then follows by inserting this bound into \eqref{eq:splitting up S(k(X))}.\end{proof}
\par The relation \eqref{eq:splitting for U} implies that
\[
S(\mathfrak{U})\leq\sum_{0\leq l\leq c^{'}\log P}S(\mathfrak{k}(2^{-l}P^{6/5})).\]
Applying Lemma \ref{lem:estimate for S(k(X)) (extended)} to each
term in the sum gives
\begin{align}
S(\mathfrak{U})&\ll \sum_{0\leq l\leq c^{'}\log P}\Big(P^{7/3}R^{13/6-2\tau/3}(2^{-l}P^{6/5})^{\varepsilon-1/3}+(2^{-l}P^{6/5})P^{1+\varepsilon}R^2\Big)\\ &\ll P^{29/15+\varepsilon}R^{13/6-2\tau/3}+P^{11/5+\varepsilon}R^2.\end{align}
The constraint $\theta\leq1/3$ implies that the second term here dominates. This is summarised by the following proposition.
\begin{prop}
\label{pro:S(U)}Whenever $\theta$ is a real number with $\theta\leq2/9$,
we have
\[
S(\mathfrak{U})\ll P^{11/5+\varepsilon}R^2.\]
\end{prop}
In particular, when $\theta>1/5$, the bound here is $o(PR^4)$.
\par For the treatment of $S(\mathfrak{A})$ in \eqref{eq:splitting for existence},
we apply Bessel's inequality and H\"{o}lder's inequality to get
\begin{align}
S(\mathfrak{A})\leq & \int_{\mathfrak{A}}|f^{*}(\alpha)^{2}K(\alpha)^{2}h(\alpha)^{4}|\:\mathrm{d}\alpha\nonumber \\
\leq & U_2^{1/12}\Big(\int_{\mathfrak{M}(R)}|f^*(\alpha)^{7/3}h(\alpha)^6|\;\mathrm{d}\alpha\Big)^{2/3}U_1^{1/4}.\label{eq:Holder on S(A)}\end{align}
According to \eqref{eq:U_3} and \eqref{eq:f}, the integral here is just $U_3(7/3,R;\mathcal{B})$ with $\mathcal{B}=\mathcal{A}(R)$. Applying Lemmata \ref{lem:mean of f*^8 and f*^16/3}, \ref{lem:mean of f*^vh^6}
and \eqref{eq:|K|^8} on \eqref{eq:Holder on S(A)} immediately
gives
\[
S(\mathfrak{A})\ll(P^{7/3}L^{-4/9})^{1/12}(P^{-2/3}R^{6})^{2/3}(P^{5})^{1/4},\]
the result of which is stated in the following proposition.
\begin{prop}
\label{pro:S(A)}For any positive number $\theta$ with $\theta\leq1/3$,
we have
\[
S(\mathfrak{A})\ll PR^4L^{-1/27}.\]
\end{prop}
Propositions \ref{pro:S(m)}, \ref{pro:S(D)}, \ref{pro:S(U)} and
\ref{pro:S(A)}, together with \eqref{eq:splitting for existence},
thus imply the following.
\begin{prop}
Let $\theta$ be a real number satisfying $192/869\leq\theta\leq2/9$.
Then
\[
S(\mathfrak{p})\ll PR^4L^{-1/50}.\]
\end{prop}
\par A simple averaging argument then reveals that for all except $O(NL^{-1/100})$
integers $n$ with $N<n\leq2N$, we have
\[
\rho_{\theta}(n;\mathfrak{p})\ll P^{-1}R^{2}L^{-1/200}.\]
This together with \eqref{eq:parameters for existence} yield the following proposition.
\begin{prop}
\label{pro:minor arc estimate for existence}Whenever $\theta$ is
a real number satisfying $192/869\leq\theta\leq2/9$, we have
\[
\rho_{\theta}(n;\mathfrak{p})\ll n^{2\theta-1/3}(\log n)^{-1/20}\]
for all integers $n$ with $N<n\leq2N$, with at most $O(N(\log N)^{-1/10})$
exceptions.
\end{prop}
Theorem \ref{thm:existence} follows from Propositions \ref{pro:major arc estimate for existence}
and \ref{pro:minor arc estimate for existence}, recalling that $r_{\theta}(n)\geq\rho_{\theta}(n)$,
and finally by summing over dyadic intervals.
\section{\label{sec:An-asymptotic-formula}The asymptotic formula}
Our goal in this section is to establish Theorem \ref{thm:asymptotic}. Let $N$ be a large integer, $P=(N/4)^{1/3}$, and $R$ be a parameter in the interval $[N^{\theta},(2N)^{\theta}]$. Observe that if $n$ is the integer in \eqref{eq:r(n)} with $N<n\leq2N$, then at least one of $x_{1},x_{2},y_{1}$ and $y_{2}$ is greater
than $P$. We know from the restrictions that $y_{1},y_{2}\leq n^{\theta}$
with $\theta\leq1/3$ that neither $y_{1}^{3}$ nor $y_{2}^{3}$ exceeds
$P$. So one of $x_{1}$ and $x_{2}$ is greater than $P$. For all
integers $ $$n$ with $N<n\leq2N$, we thus define $\sigma_{\theta}(n)$
to be the number of solutions to \eqref{eq:r(n)} with
\begin{equation}
1\leq x_{1},x_{2}\leq2P,\qquad\max\{x_{1},x_{2}\}>P,\qquad1\leq y_{1},y_{2}\leq R.\label{eq:parameters in asymptotic formula}\end{equation}
When $\alpha\in[0,1)$, write
\begin{equation}
F(\alpha)=\sum_{1\leq x\leq2P}e(\alpha x^{3}),\qquad F_{0}(\alpha)=\sum_{1\leq x\leq P}e(\alpha x^{3})\label{eq:F}\end{equation}
and 
\begin{equation}
G(\alpha)=\sum_{1\leq y\leq R}e(\alpha y^{3}).\label{eq:G}\end{equation}
For any measurable set $\mathfrak{B}\subseteq[0,1)$, write
\begin{equation}
\sigma_{\theta}(n;\mathfrak{B})=\int_{\mathfrak{B}}\Big(F(\alpha)^{2}-F_{0}(\alpha)^{2}\Big)G(\alpha)^{2}e(-n\alpha)\mathrm{d}\alpha.\label{eq:sigma(n,B)}\end{equation}
Then by orthogonality, we have $\sigma_{\theta}(n)=\sigma_{\theta}(n;[0,1))$
for all integers $n$ with $N<n\leq2N$.
\par Take $L=(\log P)^{100}.$ Recall the respective definitions \eqref{eq:P(q,a)}
and \eqref{eq:M(q,a;X)} of $\mathfrak{P}$ and $\mathfrak{M}(X)$.
For all integers $a$ and $q$ with $0\leq a\leq q\leq P^{3/4}$ and
$(a,q)=1$, introduce the major arc
\begin{equation}
\mathfrak{N}(q,a)=\mathfrak{M}(q,a;P^{3/4}).\label{eq:N(q,a)}\end{equation}
Write $\mathfrak{N}$ for the union of all these major arcs $\mathfrak{N}(q,a)$,
and $\mathfrak{n}=[0,1)\backslash\mathfrak{N}$ be the corresponding
minor arc.
\section{Major arc estimate}
As in section \ref{sec:Existence}, we first seek approximations for
the generating functions in the integral defining $\sigma_{\theta}(n;\mathfrak{P})$.
Recall the definition \eqref{eq:w} of $w(\beta;Z)$ in section 4. By Theorem 4.1 of \cite{vaughan1997}, when $\alpha\in[0,1)$, $a\in\mathbb{Z}$
and $q\in\mathbb{N}$ with $(a,q)=1$, we have the relation
\begin{equation}
F(\alpha)=q^{-1}S(q,a)w(\alpha-a/q;2P)+O(q^{1/2+\varepsilon}(1+P^{3}|\alpha-a/q|)^{1/2}).\label{eq:approximating F}\end{equation}
In particular, whenever $\alpha\in\mathfrak{P}(q,a)\subseteq\mathfrak{P}$, we get
\begin{equation}
F(\alpha)=q^{-1}S(q,a)w(\alpha-a/q;2P)+O(L^{1/2+\varepsilon}).\label{eq:approx. for F on P(q,a)}\end{equation}
By the same token, when $\alpha\in\mathfrak{P}(q,a)$, we have
\begin{equation}
F_{0}(\alpha)=q^{-1}S(q,a)w(\alpha-a/q;P)+O(L^{1/2+\varepsilon})\label{eq:approx. for F_0}\end{equation}
and \begin{equation}
G(\alpha)=q^{-1}S(q,a)w(\alpha-a/q;R)+O(L^{1/2+\varepsilon}).\label{eq:approx. for G}\end{equation}
When $\beta$ is a real number, write
\begin{equation}
W(\beta)=(w(\beta;2P)^{2}-w(\beta;P)^{2})w(\beta;R)^{2}.\label{eq:W}\end{equation}
We deduce from \eqref{eq:approx. for F on P(q,a)}, \eqref{eq:approx. for F_0}
and \eqref{eq:approx. for G} that
\begin{equation}
(F(\alpha)^{2}-F_{0}(\alpha)^{2})G(\alpha)^{2}=(q^{-1}S(q,a))^{4}W(\alpha-a/q)+O(P^{2}L^{1+\varepsilon}).\label{eq:before integrating in asymptotic}\end{equation}
\par Define $\mathfrak{S}(n;L)$ and $A(q,n)$ as in \eqref{eq:truncated singular series}
and \eqref{eq:A(q,n)} respectively, and write
\begin{equation}
\mathcal{J}(n;L)=\int_{-L/N}^{L/N}W(\beta)e(-n\beta)\:\mathrm{d}\beta.\label{eq:truncated singular integral for asymptotic}\end{equation}
Recall that the measure $\mathfrak{P}$ is $O(L^{3}/N)$. Integrating
both sides of \eqref{eq:before integrating in asymptotic} against
$e(-n\alpha)$ over all $\alpha\in\mathfrak{P}$ thus gives
\begin{equation}
\sigma_{\theta}(n;\mathfrak{P})=\mathfrak{S}(n;L)\mathcal{J}(n;L)+O(P^{-1}L^{4+\varepsilon}).\label{eq:truncated asymptotic for asymptotic}\end{equation}
Recall the definition \eqref{eq:v(beta)} of $v(\beta;Z)$. From \eqref{eq:W},
we get
\begin{align}
W(\beta)&=\Big((v(\beta;P)+w(\beta;P))^{2}-w(\beta;P)^{2}\Big)w(\beta;R)^{2}\nonumber\\&=(v(\beta;P)^{2}+2v(\beta;P)w(\beta;P))w(\beta;R)^{2}.\label{eq:breaking apart W(beta)}\end{align}
With this together with \eqref{eq:estimate for w_0}, \eqref{eq:estimate for v(beta)} and a trivial
estimate for $w(\beta;R)$, we get
\[
W(\beta)\ll P^{2}R^{2}(1+P^{3}|\beta|)^{-4/3}\]
for all real $\beta$. This confirms the absolute and uniform
convergence over $n$ of the singular integral
\begin{equation}
\mathcal{J}(n)=\int_{-\infty}^{\infty}W(\beta)e(-n\beta)\:\mathrm{d}\beta.\label{eq:singular integral for asymptotic}\end{equation}
Also, for all positive integers $n$, we have
\begin{equation}
|\mathcal{J}(n)-\mathcal{J}(n;L)|\ll P^{2}R^{2}\int_{L/N}^{\infty}(1+P^{3}\beta)^{-4/3}\mathrm{d}\beta\ll P^{-1}R^{2}L^{-1/3}.\label{eq:tail of singular integral for asymptotic}\end{equation}
\par The value of the singular integral can be computed as follows. Putting \eqref{eq:breaking apart W(beta)} into \eqref{eq:singular integral for asymptotic} gives
\[
\mathcal{J}(n)=\int_{-\infty}^{\infty}(v(\beta;P)^{2}+2v(\beta;P)w(\beta;P))w(\beta;R)^{2}e(-n\beta)\;\mathrm{d}\beta.\]
By a change of variables, when $Z$ is any positive number and $\beta$ is real, we have
\[
v(\beta;Z)=Zv(\beta Z^3;1)\qquad\text{and}\qquad w(\beta;Z)=Zw(\beta Z^3;1).\]
These two equalities imply that
\begin{equation}
\mathcal{J}(n)=P^{-1}R^2\int_{-\infty}^{\infty}\Big(v(\beta;1)^2+2v(\beta;1)w(\beta;1)\Big)w(\beta R^3P^{-3};1)^2e(-n\beta P^{-3})\;\mathrm{d}\beta.\label{eq:J(n) before error term appears}\end{equation}
Using a first order Taylor approximation, we have
\[
w(\xi;1)=1+O(\min\{1,|\xi|\})\]
for all real $\xi$. Meanwhile, the inequalities \eqref{eq:estimate for v(beta)} and \eqref{eq:estimate for w_0} imply that
\[
v(\beta;1)^2+2v(\beta;1)w(\beta;1)\ll|\beta|^{-4/3}\]
for all real $\beta$. Coupled with \eqref{eq:J(n) before error term appears}, these two facts give rise to the equality
\begin{equation}
\mathcal{J}(n)=P^{-1}R^2(\mathcal{J}^*(n)+E),\label{eq:preliminary asymptotic for J(n)}\end{equation}
where
\begin{equation}
\mathcal{J}^*(n)=\int_{-\infty}^{\infty}\Big(v(\beta;1)^2+2v(\beta;1)w(\beta;1)\Big)e(-n\beta P^{-3})\;\mathrm{d}\beta,\label{eq:J*(n)}\end{equation}
and
\begin{equation}
E\ll\int_{0}^{\infty}\beta^{-4/3}\min\{1,\beta R^3P^{-3}\}\;\mathrm{d}\beta\ll P^{-1}R.\label{eq:error term in J(n)}\end{equation}
According to p. 21-22 of \cite{davenport2005}, we have $\mathcal{J}^*(n)=\Gamma(4/3)^{2}\Gamma(2/3)^{-1}$. This, together with \eqref{eq:preliminary asymptotic for J(n)}, \eqref{eq:J*(n)}, \eqref{eq:error term in J(n)} and our choice of parameters, gives rise to the asymptotic formula
\[
\mathcal{J}(n)=\frac{\Gamma(4/3)^{2}}{\Gamma(2/3)}n^{2\theta-1/3}(1+O(n^{\theta-1/3})).\]
\par Meanwhile, the singular series $\mathfrak{S}(n)$ as defined in \eqref{eq:singular series}
stays absolutely and uniformly convergent. The estimate \eqref{eq:tail of singular series}
remains true for all but $O(NL^{-1/16})$ integers $n$ with $N<n\leq2N$. This together with
\eqref{eq:tail of singular integral for asymptotic} and \eqref{eq:truncated asymptotic for asymptotic}
lead to the following proposition.
\begin{prop}
\label{pro:major arc estimate for asymptotic}Let $\theta$ be a positive
number with $\theta\leq1/3$. Then for all but $O(N(\log N){}^{-6})$
integers $n$ with $N<n\leq2N$, we have
\begin{equation}
\sigma_{\theta}(n;\mathfrak{P})=\frac{\Gamma(4/3)^{2}}{\Gamma(2/3)}\mathfrak{S}(n)P^{-1}R^{2}+O(P^{-1}R^{2}(\log N)^{-6}).\label{eq:major arc estimate for asymptotic}\end{equation}
\end{prop}
\section{Minor arc estimate}
As in section \ref{sec:Existence}, we wish to obtain a bound for
the mean square value
\[
\sum_{N<n\leq2N}|\sigma_{\theta}(n;\mathfrak{p})|^{2}\]
which is $o(PR^{4})$. When $\mathfrak{B}$ is a measurable subset
of $[0,1)$, write
\begin{equation}
\Xi(\mathfrak{B})=\int_{\mathfrak{B}}|(F(\alpha)^{2}-F_{0}(\alpha)^{2})^{2}G(\alpha)^{4}|\:\mathrm{d}\alpha.\label{eq:Xi(B)}\end{equation}
An application of Bessel's inequality gives
\begin{equation}
\sum_{N<n\leq2N}|\sigma_{\theta}(n;\mathfrak{p})|^{2}\leq\Xi(\mathfrak{p})\label{eq:Bessel for asymptotic}\end{equation}
We then make use of the pruning $\mathfrak{p}=\mathfrak{n}\cup(\mathfrak{N}\backslash\mathfrak{P})$
and the splitting
\begin{equation}
\Xi(\mathfrak{p})=\Xi(\mathfrak{n})+\Xi(\mathfrak{N}\backslash\mathfrak{P}).\label{eq:pruning for asymptotic}\end{equation}
\par Define $f$ as in section \ref{sec:Existence}. From the factorisation
\[
F(\alpha)^{2}-F_{0}(\alpha)^{2}=f(\alpha)(F(\alpha)+F_{0}(\alpha)),\]
we can expand $\Xi(\mathfrak{n})$ as
\begin{equation}
\Xi(\mathfrak{n})\ll\int_{\mathfrak{n}}|f(\alpha)^{2}F(\alpha)^{2}G(\alpha)^{4}|\:\mathrm{d}\alpha+\int_{\mathfrak{n}}|f(\alpha)^{2}F_{0}(\alpha)^{2}G(\alpha)^{4}|\:\mathrm{d}\alpha.\label{eq:expanding Xi(n)}\end{equation}
An application of H\"{o}lder's inequality yields
\begin{equation}
\int_{\mathfrak{n}}|f(\alpha)^{2}F(\alpha)^{2}G(\alpha)^{4}|\:\mathrm{d}\alpha\leq\Big(\sup_{\alpha\in\mathfrak{n}}|f(\alpha)|\Big)^{2}\int_{0}^{1}|F(\alpha)^{2}G(\alpha)^{4}|\:\mathrm{d}\alpha.\label{eq:splitting up Xi(n)}\end{equation}
A modified version of Weyl's inequality (see, for instance, Lemma
1 of \cite{vaughan1986}) confirms that
\begin{equation}
\sup_{\alpha\in\mathfrak{n}}|f(\alpha)|\ll P^{3/4+\varepsilon}.\label{eq:Weyl's ineq. on f over n}\end{equation}
Meanwhile, Lemma \ref{lem:mean value of f^2g^4} gives the bound
\begin{equation}
\int_{0}^{1}|F(\alpha)^{2}G(\alpha)^{4}|\:\mathrm{d}\alpha\ll P^{\varepsilon}(PR^{2}+P^{-1}R^{9/2}).\label{eq:mean value for F^2G^4}\end{equation}
Hence \eqref{eq:splitting up Xi(n)}, \eqref{eq:Weyl's ineq. on f over n}
and \eqref{eq:mean value for F^2G^4} imply that
\begin{align*}
\int_{\mathfrak{n}}|f(\alpha)^{2}F(\alpha)^{2}G(\alpha)^{4}|\:\mathrm{d}\alpha&\ll(P^{3/4})^{2}P^{\varepsilon}(PR^{2}+P^{-1}R^{9/2})\\
&=P^{\varepsilon}(P^{5/2}R^{2}+P^{1/2}R^{9/2}).\end{align*}
An almost identical argument yields the same upper bound for the other integral in \eqref{eq:expanding Xi(n)}, whence
\begin{equation}
\Xi(\mathfrak{n})\ll P^{\varepsilon}(P^{5/2}R^{2}+P^{1/2}R^{9/2}).\label{eq:estimate for Xi(n)}\end{equation}
\par As for the contribution from the set of arcs $\mathfrak{N}\backslash\mathfrak{P}$,
note that expanding in \eqref{eq:Xi(B)} gives
\begin{equation}
\Xi(\mathfrak{N}\backslash\mathfrak{P})\ll\int_{\mathfrak{N}\backslash\mathfrak{P}}|F(\alpha)^{4}G(\alpha)^{4}|\:\mathrm{d}\alpha+\int_{\mathfrak{N}\backslash\mathfrak{P}}|F_{0}(\alpha)^{4}G(\alpha)^{4}|\:\mathrm{d}\alpha.\label{eq:expanding Xi(N-P)}\end{equation}
Recalling the respective definitions \eqref{eq:F*} and \eqref{eq:N(q,a)} of $F^*$ and $\mathfrak{N}(q,a)$, we deduce from \eqref{eq:approximating F} that
\[ F(\alpha)=F^*(\alpha;P^{3/4})+O(P^{3/8+\varepsilon})\]
for all $\alpha\in\mathfrak{N}$. An application of Hua's lemma thus yields
\begin{align}
&\int_{\mathfrak{N}\backslash\mathfrak{P}}|F(\alpha)^{4}G(\alpha)^{4}|\:\mathrm{d}\alpha\nonumber\\
=&\int_{\mathfrak{N}\backslash\mathfrak{P}}|F^*(\alpha;P^{3/4})^{4}G(\alpha)^{4}|\:\mathrm{d}\alpha+O\Big(P^{3/2+\varepsilon}\int_{0}^{1}|G(\alpha)|^4\;\mathrm{d}\alpha\Big)\nonumber\\
=&\int_{\mathfrak{N}\backslash\mathfrak{P}}|F^*(\alpha;P^{3/4})^{4}G(\alpha)^{4}|\:\mathrm{d}\alpha+O(P^{3/2+\varepsilon}R^2).\label{eq:replacing F by F*}\end{align}
\par An application of H\"{o}lder's inequality yields
\begin{align}
\int_{\mathfrak{N}\backslash\mathfrak{P}}&|F^*(\alpha;P^{3/4})^{4}G(\alpha)^{4}|\:\mathrm{d}\alpha\nonumber\\
\leq&\Big(\int_{\mathfrak{N}}|F^*(\alpha;P^{3/4})^{7/2}G(\alpha)^6|\;\mathrm{d}\alpha\Big)^{2/3}\Big(\int_{\mathfrak{N}\backslash\mathfrak{P}}|F^*(\alpha)|^{5}\:\mathrm{d}\alpha\Big)^{1/3}.\label{eq:Holder on int(F^4G^4) over N without P}\end{align}
From \eqref{eq:U_4} and \eqref{eq:G}, the first integral here is $U_4(7/2,P^{3/4};\mathcal{B})$ with $\mathcal{B}=[1,R]\cap\mathbb{Z}$. An application of Lemma \ref{lem:mean of f*^vh^6} thus yields the estimate
\begin{equation} \int_{\mathfrak{N}}|F^*(\alpha;P^{3/4})^{7/2}G(\alpha)^6|\:\mathrm{d}\alpha\ll P^{1/2}R^6.\label{eq:F*^7/2G^6}\end{equation}
Now the argument of Lemma 5.1 of \cite{vaughan1989} provides the bound
\begin{equation}
\int_{\mathfrak{N}\backslash\mathfrak{P}}|F^*(\alpha;P^{3/4})|^{5}\:\mathrm{d}\alpha\ll P^{2}L^{-1}.\label{eq:int(F^6) over N without P}\end{equation}
Inserting the estimates \eqref{eq:F*^7/2G^6} and \eqref{eq:int(F^6) over N without P}
into \eqref{eq:Holder on int(F^4G^4) over N without P}, and using \eqref{eq:replacing F by F*}, we obtain
\[
\int_{\mathfrak{N}\backslash\mathfrak{P}}|F(\alpha)^{4}G(\alpha)^{4}|\:\mathrm{d}\alpha\ll(P^{1/2}R^6)^{2/3}(P^2L^{-1})^{1/3}+P^{3/2+\varepsilon}R^2\ll PR^4L^{-1/3}.\]
A parallel argument yields the same upper bound for the remaining integral in \eqref{eq:expanding Xi(N-P)}. We therefore obtain
\[
\Xi(\mathfrak{N}\backslash\mathfrak{P})\ll PR^4L^{-1/3}.\]
\par This together with \eqref{eq:Bessel for asymptotic}, \eqref{eq:pruning for asymptotic}
and \eqref{eq:estimate for Xi(n)} imply the following.
\begin{prop}
As long as $1/4<\theta<1/3$, we have
\[
\sum_{N<n\leq2N}|\sigma_{\theta}(n;\mathfrak{p})|^{2}\ll PR^{4}L^{-1/25}.\]
\end{prop}
We conclude this section with a proof of Theorem \ref{thm:asymptotic}.
A simple averaging argument from the above Proposition implies that
for such $\theta$, the inequality
\begin{equation}
|\sigma_{\theta}(n;\mathfrak{p})|\ll P^{-1}R^{2}L^{-1/100}\ll P^{-1}R^{2}(\log N)^{-1}\label{eq:averaging on sigma(n;p)}\end{equation}
holds for all integers $n$ with $N<n\leq N+N(\log N)^{-2}$, with
$O(N(\log N)^{-4})$ exceptions. Coupled with
the conclusion of Proposition \ref{pro:major arc estimate for asymptotic},
the upper bound \eqref{eq:averaging on sigma(n;p)} implies that
\begin{equation}
\sigma_{\theta}(n)=\sigma_{\theta}(n;[0,1))=\frac{\Gamma(4/3)^{2}}{\Gamma(2/3)}\mathfrak{S}(n)P^{-1}R^{2}+O(P^{-1}R^{2}(\log N)^{-1})\label{eq:preliminary asymptotic for asymptotic}\end{equation}
for all but $O(N(\log N)^{-4})$ integers $n$ with $N<n\leq N+N(\log N)^{-2}$.
For all such $n$ and $\theta$, we have
\[
N^{\theta}<n^{\theta}\leq N^{\theta}+O(N^{\theta}(\log N)^{-2}).\]
Whence there exists a positive constant $A$ such that
\[
N^{\theta}<n^{\theta}\leq N^{\theta}+AN^{\theta}(\log N)^{-2}.\]
Recall the definition \eqref{eq:r(n)} of $r_{\theta}(n)$ as well
as equation (1.3) of \cite{kawada1996}, which gives $1\ll\mathfrak{S}(n)\ll(\log\log n)^{4}$.
Putting first $R=N^{\theta}$ and then $R=N^{\theta}+AN^{\theta}(\log N)^{-2}$
into \eqref{eq:preliminary asymptotic for asymptotic}, we arrive
at the relation
\[
r_{\theta}(n)=\frac{\Gamma(4/3)^{2}}{\Gamma(2/3)}\mathfrak{S}(n)N^{2\theta-1/3}+O(N^{2\theta-1/3}(\log N)^{\varepsilon-2}).\]
Hence when $1/4<\theta<1/3$, the aysmptotic formula\[
r_{\theta}(n)=\frac{\Gamma(4/3)^{2}}{\Gamma(2/3)}\mathfrak{S}(n)n^{2\theta-1/3}+O(n^{2\theta-1/3}(\log n)^{-1})\]
holds for all but $O(N(\log N)^{-4})$ integers $n$ with $N<n\leq N+N(\log N)^{-2}$.
There are altogether $O((\log N)^{3})$ intervals of such type that
cover $[1,2N]$. Summing over all such intervals leads to the same
asymptotic formula with the total number of exceptions encountered
being $O(N(\log N)^{-1})$. The conclusion of Theorem \ref{thm:asymptotic}
thus holds for all real numbers $\theta$ in the range $(1/4,1/3)$.\[
\]

\end{document}